\documentclass[12pt]{article}
\usepackage{latexsym, amsmath, amsfonts, amsthm, amssymb}
\usepackage{times}
\usepackage{a4wide}
\usepackage{times}
\usepackage{graphicx, tocloft}
\usepackage{mathrsfs}

\def \RR {\mathbb R}

\def \ZZ {\mathbb Z}

\def \eps {\varepsilon}

\def \cM {\mathcal M}

\def \cE {\mathcal E}

\newtheorem{theorem}{Theorem}[section]

\newtheorem{lemma}[theorem]{Lemma}

\newtheorem{proposition}[theorem]{Proposition}

\theoremstyle{definition}

\def\myffrac#1#2 in #3{\raise 2.6pt\hbox{$#3 #1$}\mkern-1.5mu\raise 0.8pt\hbox{$
		#3/$}\mkern-1.1mu\lower 1.5pt\hbox{$#3 #2$}}

\def\qed{\hfill $\vcenter{\hrule height .3mm
		\hbox {\vrule width .3mm height 2.1mm \kern 2mm \vrule width .3mm
			height 2.1mm} \hrule height .3mm}$ \bigskip}

\def \susbeteq {\subseteq}

\begin{document}

\title{A weak version of the $\eps$-Dvoretzky conjecture for normed spaces}
\author{Bo'az Klartag and Tomer Novikov}
\date{}
\maketitle

\abstract{We prove a
weak version of the $\eps$-Dvoretzky conjecture for normed spaces, showing the existence of
a subspace of $\RR^n$ of dimension at least $c \log n / |\log \eps|$ in which the given norm is $\eps$-close to a norm obeying a large discrete group of symmetries (``$1$-unconditional norm'').}

\section{Introduction}

Dvoretzky's theorem \cite{DvoretzkyOriginalAnnouncement, DvoretskyOriginal} is a fundamental result in the theory of high-dimensional normed spaces that was proved
circa 1960. Conjectured earlier by Grothendieck \cite{Grothendieck} and by others (see \cite{DvoretzkyOriginalAnnouncement}),
this theorem is formulated as follows:

\begin{theorem}[Dvoretzky's theorem] Let $||\cdot||$ be a norm in $\mathbb{R}^n$ and let $0 < \eps < 1/2$. Suppose that $k$ is an integer that satisfies
\begin{equation}
n \geq \exp \left(C k  \frac{|\log \eps|}{\eps} \right),
\label{eq_1035} \end{equation}
where $C > 0$ is a certain universal constant. Then there exists a $k$-dimensional subspace $E \subseteq \RR^n$ and $r > 0$ such that for all $x \in E$,
	$$(1-\eps) r |x|\leq||x|| \leq (1+\eps) r |x|$$
	where $| \cdot|$ is the standard Euclidean norm in $\RR^n$.
	\label{thm_1046}
\end{theorem}

Theorem \ref{thm_1046} can be reformulated as stating that any
centrally-symmetric convex body $K \subseteq \RR^n$
has a central $k$-dimensional section that is nearly spherical.

\medskip The estimate (\ref{eq_1035}) is taken from Paouris and Valettas \cite{PV}, and it improves upon an earlier
bound by Schechtman \cite{schechtman} by a factor that is logarithmic in $\eps$, which in turn improves upon an estimate by Gordon \cite{G} by a factor which is nearly $1 / \eps$. All of these proofs  utilize a highly influential approach
by V. Milman to Dvoretzky's theorem which emphasizes the role
of the concentration of measure phenomena, and which has found applications in several branches of mathematics  \cite{Mil71, ConjectureStatement}.

\medskip However, the dependence on $1/\eps$
in the estimate (\ref{eq_1035}) is exponential.
It was suggested by Milman \cite{MilmanTopology} that the actual dependence on $\eps$ in Dvoretzky's theorem should perhaps be polynomial. This is the conjectural {\it almost-isometric} variant of Dvoretzky's theorem. This conjecture holds true in the case when $k = 2$ (see \cite{MilmanTopology}) or when the norm $\| \cdot \|$ is assumed $1$-symmetric, i.e., when
$$ \left \| (\pm x_{\pi(1)},\ldots,\pm x_{\pi(n)}) \right \|
= \left \| (x_1,\ldots,x_n) \right \| $$
for all vectors $(x_1,\ldots,x_n) \in \RR^n$, for any permutation $\pi \in S_n$ and for any choice of signs.
See Bourgain-Lindenstrauss \cite{Bourgain}, Tikhomirov \cite{Tikhomirov} and Fresen \cite{Fresen} for analysis of the $1$-symmetric case.

\medskip All proofs of Theorem \ref{thm_1046} for general $k$ rely heavily on
probability and analysis, for example through the use of P. L\'evy's concentration inequality for Lipschitz functions on the high-dimensional sphere
\cite{Mil_levy}.
On the other hand, the proof for the case $k = 2$, which is attributed to Gromov by  Milman \cite{MilmanTopology}, is purely topological.
As for the case $k \geq 3$, there is hope that a proof employing topological tools could lead to better dependence on $\eps$ in Theorem \ref{thm_1046}.
See also Burago, Ivanov and Tabachnikov \cite{BBT} for a discussion of possible topological approaches to Dvoretzky's theorem
 and their shortcomings.

\medskip Our main result relies on elementary topological tools in order to make advances towards the almost-isometric variant of Dvoretzky's theorem. Unfortunately, we do not
obtain a full rotational symmetry but only the so-called ``$1$-unconditional symmetries''.
A norm $\| \cdot \|$ in $\RR^n$ is ``unconditional'' with respect to the orthonormal basis $(e_1,\ldots,e_n)$ if for any $x_1,\ldots,x_n \in \RR$ and any choice of signs,
$$ \left \| \sum_{i=1}^n \pm x_i e_i \right \| = \left \|  \sum_{i=1}^n x_i e_i \right \|. $$
Thus the norm $\| \cdot \|$ admits a symmetry group with $2^n$ elements.
We say that a norm defined in a subspace of $\RR^n$ is unconditional if there exists an orthonormal basis in this subspace with respect to which it is unconditional.

\begin{theorem} Let $||\cdot||$ be a norm in $\mathbb{R}^n$ and let $0 < \eps < 1/2$. Suppose that $k \geq 2$ is an integer that satisfies
	\begin{equation}
	n \geq \left(\frac{C}{\eps} \right)^{3 (k-1)}
	\label{eq_1035_} \end{equation}
	for a certain universal constant $C > 0$. Then there exists a $k$-dimensional subspace $E \subseteq \RR^n$ and an unconditional norm $\|| \cdot \|| $ in the subspace $E$ such that for all $x \in E$,
	$$(1-\eps) \|| x \|| \leq||x|| \leq (1+\eps) \||x \||. $$
\label{thm3}
\end{theorem}

While the dependence on $1/\eps$ in Theorem \ref{thm3}
is only polynomial, as desired, the exponent $3(k-1)$ in (\ref{eq_1035_}) is non-optimal.
It may be replaced by $\alpha (k-1)$ for any $\alpha > 2$
at the expense of modifying the value of the universal constant $C > 0$ from Theorem \ref{thm3}, as can be seen from the proof.
It is likely possible to replace the ``unconditional'' symmetries of the norm $\|| \cdot \||$ in Theorem \ref{thm3} by other commutative groups of symmetries, such as cyclic permutations of the coordinates.
See also Makeev \cite{makeev} and Burago, Ivanov and Tabachnikov \cite{BBT}.
However, we do not know how to obtain a discrete group of symmetries such as the group of permutations $S_n$, or its $2$-Sylow subgroup, while keeping the dependence on $1 / \eps$ polynomial.

\medskip Theorem \ref{thm_1046} and Theorem \ref{thm3} should be compared with the results
of Alon and Milman \cite{AM} that show in particular that any normed space of sufficiently high dimension contains a $k$-dimensional subspace that is $\eps$-close
either to $\ell_2^k$ or to $\ell_{\infty}^k$.
The dependence on $\eps$ here is  worse than that of Theorem \ref{thm3}, and it seems to us that the results of \cite{AM} are not directly applicable in the case where, say, $\eps \leq 1 / \log^2 n$
where $n$ is the dimension of the given normed space. On the other hand, when $\eps$ is a fixed constant that is not allowed to depend on the dimension, 
Alon and Milman \cite{AM} yield a stronger conclusion than that of Theorem \ref{thm3}, allowing $k$ to be as large as $\exp(c \sqrt{\log n})$
and yielding a subspace isomorphic to $\ell_2^k$ or to $\ell_{\infty}^k$.
The argument from \cite{AM} was simplified by Talagrand \cite{Tal} and the results of \cite{AM} were used by Schechtman \cite{schechtman} and by
Paouris and Valettas \cite{PV} in their proof of Theorem \ref{thm_1046}.

\medskip In this paper we write $\langle \cdot, \cdot \rangle$
for the standard Euclidean product in $\RR^n$, and $|x| = \sqrt{\langle x, x \rangle}$. By $\log$ we refer to the natural logarithm. Throughout this text we use the letters $c, C, \tilde{C}$ etc. to denote positive universal constants, that may be explicitely computed in principle, whose value may change from one line to the next.

\medskip {\it Acknowledgements.}
Supported by a  grant from the Israel Science Foundation (ISF).
This paper is based on an M.Sc. thesis written by the second named author under the supervision of the first named author
at the Weizmann Institute of Science. We thank Roman Karasev for noting an error in an earlier version of this manuscript.

\section{Unconditional symmetries}
\label{sec3}

Consider the group $(\ZZ / 2 \ZZ)^k \cong \{ \pm 1 \}^k$.
This group acts on the Stiefel manifold $W_{n,k}$ of $k$-frames in $\RR^n$ by switching the signs of the frame vectors, i.e.,
\begin{equation}  g.(U_1,\ldots,U_k) = (g_1 U_1,\ldots, g_k U_k) \qquad \text{for} \ (U_1,\ldots,U_k) \in W_{n,k}, (g_1,\ldots,g_k) \in \{ \pm 1 \}^k \label{eq_1148} \end{equation}
where $U_1,\ldots,U_k \in \RR^n$ is a $k$-frame.
A linear action of a group $G$ on $\RR^{\ell}$ (i.e., a
representation) is {\it fixed-point-free} if there is no vector $0 \neq x \in \RR^{\ell}$
with $g.x = x$ for all $g \in G$.

\begin{proposition} Consider any fixed-point-free
	representation of $G = \{ \pm 1 \}^k$ in $\RR^{\ell}$ for $\ell \leq n-k$.
	Then any continuous, $G$-equivariant map  $F: W_{n,k} \rightarrow \RR^{\ell}$ has to vanish somewhere in $W_{n,k}$.
	\label{prop_154}
\end{proposition}

	The difference between Proposition \ref{prop_154} and
Theorem 1.1
in Chan, Chen, Frick and Hull \cite{frick}, is that the respresentation in $\RR^{\ell}$ can be arbitrary, as long as it is fixed-point-free. We guess that there should be an elegant algebraic-topology proof of Proposition \ref{prop_154}, perhaps using Stiefel-Whitney classes.

\begin{proof}[Proof of Proposition \ref{prop_154}]
For a non-empty subset $A \subseteq \{1, \ldots,k \}$ we  define the one-dimensional representation
$$ w_A(g) = \prod_{i \in A} g_i \qquad \qquad \text{for} \ g = (g_1,\ldots,g_k) \in \{ \pm 1 \}^k, $$
where the linear action on $\RR$ is given by  $g.x = w_A(g)  x$ for $x \in \RR, g \in G$.
Any fixed-point-free, irreducible representation
of the abelian group $G = \{ \pm 1 \}^k$ is isomorphic to one of these $2^k - 1$ one-dimensional representations.
 Any finite-dimensional representation
of $G$ splits into a direct sum of irreducible representations. Thus every  fixed-point-free, finite-dimensional representation of $\{ \pm 1 \}^k$ is characterized by a formal sum
\begin{equation}  \tau = \sum_{\emptyset \neq A \subseteq \{1, \ldots,k \}} m_A \cdot A \label{eq_1718} \end{equation}
for non-negative integers $m_A$, which count the number of times that each irreducible representation occurs in the given finite-dimensional representation. Note that the dimension of the representation is $|\tau| := \sum_A |m_A|$.	

\medskip Recall that  we are given a certain fixed-point-free representation of the group $G$ in $\RR^{\ell}$. Write the formal sum corresponding to this representation as
$$ S_1 + S_2 + \ldots + S_k $$
where $S_i$ is the part of the formal sum that includes all subsets $A \subseteq \{ 1, \ldots, k \}$, where the maximal element of $A_{i,j}$ is precisely  $i$. Since $\ell \leq n-k$ we have
$$ |S_i| \leq n-k \qquad \qquad \qquad \text{for} \ i=1,\ldots,k. $$
Denote $N = nk - k(k+1)/2 = \dim(W_{n,k})$.
Consider a representation of $G$ in the space $\RR^N$ corresponding to the formal sum
\begin{equation}  \tilde{\tau} = \sum_{i=1}^k (n-i-|S_i|) \cdot \{ i \} + \sum_{i=1}^k S_i. \label{eq_436} \end{equation}
This representation in the space $\RR^N$ has an invariant subspace isomorphic to the representation in the space $\RR^{\ell}$ that is given to us. Hence it suffices to prove that any continuous, $G$-equivariant map from $X$ to $\RR^N$, vanishes somewhere in $W_{n,k}$. In view of \cite[Theorem 2.1]{K_waist}, it suffices to
construct a smooth, $G$-equivariant map $f: W_{n,k} \rightarrow \RR^{N}$, of which zero is a regular value,
such that $f^{-1}(0)$ is an orbit of $G$. The function $f$ that we will construct takes the form
$$ f = (f_{i,j})_{1 \leq i \leq k, 1 \leq j \leq n-i} $$
for scalar functions $f_{i,j}: W_{n,k} \rightarrow \RR$.
These scalar functions are defined as follows: For $1 \leq i \leq k$ and $|S_i| + 1 \leq j \leq n-i$ we set
\begin{equation}  f_{i,j}(U) = U_{i,i+j} \qquad \qquad \qquad (U \in W_{n,k}) \label{eq_434_} \end{equation}
where $U = (U_1,\ldots,U_k)$ is a $k$-frame in $\RR^n$, and $U_i = (U_{i,1},\ldots,U_{i,n}) \in \RR^n$. We still need to define $f_{i,j}$ for $1 \leq j \leq |S_i|$. Let us write
$$ S_i = \sum_{j=1}^{|S_i|} A_{i,j} $$
for non-empty subsets $A_{i,j} \subseteq \{1,\ldots,k \}$
whose maximal element  equals $i$.
	For $1 \leq j \leq |S_i|$ we set
\begin{equation} f_{i,j}(U) = U_{i,i+j} \cdot \prod_{r \in A_{i,j} \setminus \{ i \}} U_{r,r}. \label{eq_434} \end{equation}
This completes the definition of the smooth map $f: W_{n,k} \rightarrow \RR^N$. Recalling (\ref{eq_1148}),
we observe that the map $f$ is
indeed $G$-equivariant; its coordinates $f_{i,j}$ that
are defined in (\ref{eq_434_}) correspond to the first summand in (\ref{eq_436}), while its coordinates that are defined in (\ref{eq_434}) correspond to the second summand.

\medskip Let us now describe the zero set of $f$.
Suppose that $U \in W_{n,k}$ satisfies $f(U) = 0$.
The fact that $f_{1,j}(U) = 0$ for all $1 \leq j \leq n-1$ implies that $$ U_1 = e_1 = (\pm 1,0,\ldots,0). $$
Similarly, the facts that $f_{2,j}(U) = 0$ for all $1 \leq j \leq n-2$ and that $U_2 \perp U_1$ imply that $U_2 = e_2 = (0, \pm 1,0,\ldots,0)$.
By a straightforward induction argument,
we see that $U_j = \pm e_j$, where $e_1,\ldots,e_n \in \RR^n$ are the standard unit vectors. Thus $$ f^{-1}(0) = \{ (\delta_1 e_1,\ldots, \delta_k e_k) \, ; \, \delta_i = \pm 1 \ \text{for} \ i=1,\ldots,k \} \subseteq W_{n,k}. $$ We see that  $f^{-1}(0)$ is a set of size $2^k$ which is an orbit of $G$. We need to explain why zero is a regular value of $f$. To this end we consider a
smooth regular curve $U(t) \in W_{n,k}$, defined for $t \in (-1,1)$  with $U(0)  \in f^{-1}(0)$. Note that the derivatives
with respect to the variable $t$, denoted by $\dot{U}_{1,1}(t),\ldots,\dot{U}_{k,k}(t)$ all vanish for $t = 0$, because $|U_{i,i}(0)| = 1 \geq |U_{i,i}(t)|$ for all $i$ and $t$. A crucial observation is that the derivative at $t =0$ of the vector
$$ (U_{i,i+j}(t))_{1 \leq i \leq k, 1 \leq j \leq n-i}  \in \RR^N $$
does not vanish, because $\dot{U}(0) \neq 0$ and because $U(t) \in W_{n,k}$ for all $t$. It thus follows from (\ref{eq_434_}), (\ref{eq_434})
and the Leibnitz rule for differentiation that  $\frac{d}{dt} f(U(t))$ does not vanish for $t =0$, as required.
We have thus verified all of the requirements from
\cite[Theorem 2.1]{K_waist}, thereby completing the proof.
\end{proof}

Next we describe the elegant method from Barvinok \cite{barvinok} for approximating a norm by a homogeneous polynomial taken to some power. Fix $k \geq 1$.
We refer to a vector $\alpha = (\alpha_1,\ldots,\alpha_k) \in \ZZ_{\geq 0}^k$ of non-negative integers as a {\it multi-index}. The order of the multi-index $\alpha = (\alpha_1,\ldots,\alpha_k)$ is
$$ |\alpha| = \sum_{i=1}^k \alpha_i. $$
The collection of multi-indices of order $d$ is denoted by $\cM_d$. It is a set of cardinality ${d + k-1 \choose k-1}$. For $x \in \RR^k$ and $\alpha \in \cM_d$ we define
\begin{equation}  x^\alpha = \prod_{i=1}^k x_i^{\alpha_i} \in \RR.
\label{eq_1107} \end{equation}
For an integer $d \geq 1$ and $x \in \RR^k$ we set
$$ x^{\otimes d} = \left( x^\alpha \right)_{|\alpha| = d }  \in \RR^{\cM_d} $$
(the reason for this notation is  that the linear space  $\RR^{\cM_d}$ may be identified with the symmetric tensor product of order $d$ of $\RR^k$).  Note that $(-x)^{\otimes d} = (-1)^d \cdot x^{\otimes d}$.
We use $(b_{\alpha})_{\alpha \in \cM_d}$ as the coordinates of a vector $b \in \RR^{\cM_d}$, thus for instance we may write
$$ (x^{\otimes d})_{\alpha} = x^{\alpha} \qquad (x \in \RR^k, \alpha \in \cM_d). $$ For $a,b \in \RR^{\cM_d}$ we define the scalar product
$$ (a,b) = \sum_{\alpha \in \cM_d} {d \choose \alpha} a_{\alpha} b_{\alpha} $$
where ${d \choose \alpha} = d! / \prod_{i=1}^k \alpha_i!$ is a multinomial coefficient. Note that for any $x, y \in \RR^k$,
by the multinomial theorem,
\begin{equation}  (x^{\otimes d}, y^{\otimes d})
= \sum_{\alpha \in \cM_d} {d \choose \alpha} x^{\alpha} y^{\alpha} = \left( \sum_{i=1}^k x_i y_i \right)^d = \langle x,y \rangle^d. \label{eq_425} \end{equation}
 Suppose that $d$ is an odd, positive integer. Given a norm $\| \cdot \|$ in $\RR^k$, we consider the dual norm
$$ \| x \|_* = \sup_{0 \neq z \in \RR^k} \frac{\langle z, x \rangle}{\| z \|} $$
and the convex set
\begin{equation}  K_{\| \cdot \|} = conv \left \{ x^{\otimes d} \, ; \, x \in \RR^k, \ \| x \|_* \leq 1 \right \}, \label{eq_230} \end{equation}
where $conv$ denotes the convex hull of a set.
The set $K = K_{\| \cdot \|}
\subseteq \RR^{\cM_d}$ is a centrally-symmetric convex body with a non-empty interior, as explained in \cite{barvinok}.
Note that by (\ref{eq_425}) and (\ref{eq_230}), for any $x \in \RR^k$,
\begin{equation}  \| x \|^d = \sup_{\| y \|_* \leq 1} \langle x,y \rangle^d  =
\sup_{\| y \|_* \leq 1} (x^{\otimes d},y^{\otimes d})
= \sup_{a \in K} (x^{\otimes d},a). \label{eq_433} \end{equation}
Among all centrally-symmetric ellipsoids that are contained in $K = K_{\| \cdot \|}$, there is a unique ellipsoid of maximal volume (see e.g. \cite[Corollary 3.7]{pisier}). We denote this maximal-volume ellipsoid by $\cE = \cE_{\| \cdot \|} \subseteq \RR^{\cM_d}$. By the John theorem (see e.g. \cite[Chapter 3]{pisier}),
\begin{equation} \cE \subseteq K \susbeteq \sqrt{|\cM_d|} \cdot \cE, \label{eq_1103} \end{equation}
where $|\cM_d|$ is the cardinality of the set $\cM_d$.
Since $\cE$ is a centrally-symmetric ellipsoid in $\RR^{\cM_d}$ and since $(\cdot, \cdot)$ is an inner product, the expression  $ \sup_{a \in \cE} (a, b)^2$ is a quadratic function of $b \in \RR^{\cM_d}$ that is positive for all $0 \neq b \in \RR^{\cM_d}$. Consequently, there exists
a positive-definite, symmetric matrix $$ A = (A_{\alpha ,\beta})_{\alpha, \beta \in \cM_d} \in \RR^{\cM_d \times \cM_d} $$ which satisfies
\begin{equation}
\sum_{\alpha, \beta \in \cM_d} A_{\alpha, \beta} b_{\alpha} b_{\beta} = \sup_{a \in \cE} (a, b)^2 \qquad \qquad \text{for all} \ b \in \RR^{\cM_d}. \label{eq_1121} \end{equation}
 The symmetric matrix $A = A_{\| \cdot \|}$ is uniquely
 determined by the ellipsoid $\cE = \cE_{\| \cdot \|}$.
As in Barvinok \cite{barvinok}, we conclude from (\ref{eq_433}), (\ref{eq_1103}) and (\ref{eq_1121}) that
\begin{equation}   \sum_{\alpha, \beta \in \cM_d} A_{\alpha, \beta} x^{\alpha} x^{\beta} \leq \| x \|^{2d} \leq |\cM_d| \cdot \sum_{\alpha, \beta \in \cM_d} A_{\alpha, \beta} x^{\alpha} x^{\beta}. \label{eq_446}
\end{equation}
Observe also that $x^{\alpha} x^{\beta} = x^{\alpha + \beta}$ for all $x \in \RR^k, \alpha, \beta \in \cM_d$.
According to (\ref{eq_446}), the given norm $\| \cdot \|$ in $\RR^k$, taken to the power $2d$, is approximated by a $2d$-homogeneous polynomial in $\RR^k$.

\begin{lemma}
	The matrix $A = A_{\| \cdot \|}$ varies continuously with the norm $\| \cdot \|$, where we equip the space of norms on $\RR^k$ with the topology of uniform convergence on $S^{k-1}$. Moreover, for $\delta \in \{ \pm 1 \}^k$ and a norm $\| \cdot \|$ on $\RR^k$, denote
	$$ \| x \|_{\delta} = \| (\delta_1 x_1,\ldots, \delta_k x_k) \| \qquad \qquad \text{for} \ x \in \RR^k. $$
	Then the symmetric matrices  $A_{\| \cdot \|} = (A_{\alpha, \beta})_{\alpha, \beta \in \cM_d}$ and
	$A_{\| \cdot \|_{\delta}} = (A_{\alpha, \beta}(\delta))_{\alpha, \beta \in \cM_d}$ satisfy
	\begin{equation}  A_{\alpha, \beta}(\delta) = \delta^{\alpha + \beta} A_{\alpha, \beta}
\qquad \text{for} \ \alpha, \beta \in \cM_d, \label{eq_1133}
	\end{equation} \label{lem_1209}
where we recall that $\delta^{\alpha + \beta} = \prod_{i=1}^k \delta_i^{\alpha_i + \beta_i}$.
\end{lemma}

\begin{proof} Let us continuously vary the norm $\| \cdot \|$.
	Then the dual norm $\| \cdot \|$ also varies
	continuously, and the convex body $K = K_{\| \cdot \|}
	\subseteq \RR^{\cM_d}$
	varies continuously with respect to the Hausdorff metric. Since the maximal volume ellipsoid $\cE = \cE_{\| \cdot \|}$ is uniquely determined, it also varies continuously with respect to the Hausdorff metric. This follows from  the fact that if $f(x,y)$ is a continuous function of two variables in a compact metric space, and $\min_y f(x,y)$ is uniquely attained for any $x$ at a point $y_0(x)$, then $y_0(x)$ is a continuous function of $x$. The symmetric matrix $A$ is determined by $\cE$ through (\ref{eq_1121}), and it is elementary to verify its continuity. Next, for any $\delta \in \{\pm 1 \}^k$
	the convex set $ K_{\| \cdot \|_{\delta}}$ is the image of $K_{\| \cdot \|}$ under the linear map
	\begin{equation}  (y_{\alpha})_{\alpha \in \cM_d} \mapsto 	(\delta^{\alpha} y_{\alpha})_{\alpha \in \cM_d}. \label{eq_1130} \end{equation}	
	Similarly, the ellipsoid $\cE_{\| \cdot \|_{\delta}}$ is the image
of the ellipsoid $\cE_{\| \cdot \|}$ under the linear map in (\ref{eq_1130}). Hence the matrix $A_{\| \cdot \|_{\delta}}$ is
	congruent to the matrix $A_{\| \cdot \|}$ via the linear transformation (\ref{eq_1130}), and relation (\ref{eq_1133}) holds true.
\end{proof}

Consider the lexicographic order on $\cM_d$. That is,
for two distinct multi-indices $\alpha, \beta \in \cM_d$
let $i_0 \in \{ 1,\ldots,k \}$ be the minimal index with $\alpha_{i_0} \neq \beta_{i_0}$. We write that $\alpha < \beta$ if $\alpha_{i_0} < \beta_{i_0}$. It is easy to verify that $<$ is a linear order relation. Consider the subset
$$ E_d = \left \{ (\alpha, \beta) \, ; \, \alpha, \beta \in \cM_d, \ \alpha < \beta, \ \alpha + \beta \not \in 2 \ZZ^k \right \} \subseteq \cM_d \times \cM_d, $$
where $2 \ZZ^k$ is the collection of all vectors of length $k$ whose coordinates are even integers.

\begin{proof}[Proof of Theorem \ref{thm3}]
	For $U = (U_1,\ldots,U_k) \in W_{n,k}$ with $U_1,\ldots,U_k \in \RR^n$ being a $k$-frame, we define
	the norm
	$$ \| x \|_{U} = \left \| \sum_{i=1}^k x_i U_i \right \| \qquad \qquad \qquad \text{for} \ x \in \RR^k. $$
Let $d$ be an odd, positive integer such that
\begin{equation} \frac{1}{2} {d+k-1 \choose k-1}^2 \leq n-k.
\label{eq_323} \end{equation}
Let us abbreviate $A(U) := A_{\| \cdot \|_U} \in \RR^{\cM_d \times \cM_d}$. Write $A(U) = (A_{\alpha, \beta}(U))_{\alpha, \beta \in \cM_d}$
and recall that $A_{\alpha, \beta}(U) = A_{\beta, \alpha}(U)$.
Consider the map
\begin{equation}  W_{n,k} \ni U \mapsto ( A_{\alpha, \beta}(U) )_{(\alpha,\beta) \in E_d} \in \RR^{\ell} \label{eq_1209}
\end{equation}
for
\begin{equation} \ell = |E_d| \leq \frac{|\cM_d| \cdot (|M_d| - 1)}{2} \leq \frac{1}{2} {d + k-1 \choose k-1}^2 \leq  n-k. \label{eq_406} \end{equation}
The map defined in (\ref{eq_1209}) is continuous, as follows from Lemma \ref{lem_1209}. It is equivariant with respect to the group $G = \{ \pm 1 \}^k$, where the action of $G$
on $W_{n,k}$ is given by (\ref{eq_1148}),
and the action on $\RR^{\ell}$ is described by (\ref{eq_1133}). Observe that the last action is fixed-point-free, as we only consider $\alpha, \beta \in \cM_d$ with  $\alpha + \beta \not \in 2 \ZZ^k$.
We may apply Proposition \ref{prop_154} thanks to (\ref{eq_406}) and conclude that there exists $U \in W_{n,k}$ such that for
any $\alpha, \beta \in \cM_d$,
\begin{equation}  A_{\alpha, \beta}(U) = 0 \qquad \qquad \text{whenever} \ \alpha + \beta \not \in 2 \ZZ^k. \label{eq_427} \end{equation}
We fix such $U \in W_{n,k}$, and show that the norm $\| x \|_U$ is approximately an unconditional norm in $\RR^k$.
For $x \in \RR^k$ consider the $(2d)$-homogeneous polynomial,
$$ P(x) = \sum_{\alpha, \beta \in \cM_d} A_{\alpha, \beta}(U) x^{\alpha + \beta} \qquad \qquad \qquad (x \in \RR^k). $$
It follows from (\ref{eq_427}) that $P(x_1,\ldots,x_k) = P(\pm x_1,\ldots, \pm x_k)$ for any choice of signs.
In view of (\ref{eq_446}), the $\{ \pm 1 \}^k$-invariant, $(2d)$-homogeneous polynomial $P$ satisfies
\begin{equation}  P(x)^{1/(2d)} \leq \| x \|_U \leq |\cM_d|^{1/(2d)} \cdot P(x)^{1/(2d)} \qquad \qquad (x \in \RR^k). \label{eq_849} \end{equation}
Denote
\begin{equation}  \eps = |\cM_d|^{1/(2d)} - 1 \label{eq_357} \end{equation}
and set also $\|| x \|| = 2^{-n} \sum_{\delta \in \{\pm 1\}^k} \| (\delta_1 x_1,\ldots, \delta_k x_k) \|_U$. Then $\|| \cdot \||$ is an unconditional norm, and it follows from (\ref{eq_849}) that
$$ \frac{1}{1+\eps} \|| x \|| \leq \| x \|_U \leq (1 + \eps) \|| x \||
\qquad \text{for} \ x \in \RR^k. $$
Recall that $1-\eps \leq 1/(1 + \eps)$ provided that $0 < \eps < 1$.
Since $|\cM_d| = {d + k-1 \choose k-1}$ we deduce from (\ref{eq_357}) that assuming $d \geq k$,
\begin{equation} \eps \leq \left( 2 e \frac{d}{k} \right)^{\frac{k}{2d}} - 1 \leq C \frac{\log (d/k)}{d/k}
\leq \frac{\tilde{C}}{(d/k)^{2/3}}.
\label{eq_428} \end{equation}
We claim that we may choose $d \geq c k n^{\frac{1}{2(k-1)}}$
so that (\ref{eq_323}) would hold true. Indeed, we may assume that $k \leq n/2$ as otherwise there is nothing to prove. Thus (\ref{eq_323}) holds true whenever
$$ {d + k-1 \choose k-1}^2 \leq \left(2 e \frac{d}{k-1}  \right)^{2(k-1)} \leq  n. $$
Consequently, from (\ref{eq_428}), for any $k \geq 2$
the conclusion of the theorem holds true with any $0 < \eps < 1/2$ that satisfies
\begin{equation} \eps \leq C n^{-\frac{1}{3(k-1)}}
\label{eq_441} \end{equation}
for a universal constant $C > 0$. The number ``$3$'' in (\ref{eq_441}) may be replaced by any number $\alpha > 2$, at the expense of increasing the universal constant $C > 0$.	
\end{proof}

{\bf Remark.} In Bourgain-Lindenstrauss \cite{Bourgain}, Tikhomirov \cite{Tikhomirov} and Fresen \cite{Fresen},
one assumes that the given norm is invariant under the group $(\ZZ / 2 \ZZ)^n \rtimes S_n$ and concludes
the existence of approximately-spherical sections with good dependence on the degree of approximation. We note
here that it is easy to obtain $(\ZZ / 2 \ZZ)^n \rtimes S_n$-symmetric norms from $S_n$-symmetric norms,
by reducing the dimension by a factor of two. Specifically, if $\| \cdot \|$ is a permutation-invariant norm in $\RR^n$,
for even $n$, then for $m = n / 2$, the norm
$$ \|| (x_1,\ldots,x_{m}) \|| = \| (x_1, -x_1,x_2,-x_2,\ldots, x_{m}, -x_{m}) \| $$
is a $(\ZZ / 2 \ZZ)^m \rtimes S_m$-invariant norm on $\RR^m$.

\bigskip

\noindent
Department of Mathematics, Weizmann Institute of Science, Rehovot 76100, Israel. \\
{\it e-mails:} \verb"{boaz.klartag,tomer.novikov}@weizmann.ac.il"

\end{document}